\documentclass[12pt]{article}

\usepackage[usenames]{color}
\usepackage{amssymb}
\usepackage{graphicx}
\usepackage{amscd}
\usepackage{bm}
\usepackage[colorlinks=true,
linkcolor=webgreen,
filecolor=webbrown,
citecolor=webgreen]{hyperref}

\definecolor{webgreen}{rgb}{0,.5,0}
\definecolor{webbrown}{rgb}{.6,0,0}

\usepackage{color}
\usepackage{fullpage}
\usepackage{float}

\usepackage{amsmath}
\usepackage{amsthm}
\usepackage{amsfonts}
\usepackage[normalem]{ulem}
\newcommand{\stkout}[1]{\ifmmode\text{\sout{\ensuremath{#1}}}\else\sout{#1}\fi}
\newtheorem{theorem}{Theorem}
\newtheorem{lemma}{Lemma}
\newtheorem{corollary}{Corollary}

\theoremstyle{definition}\newtheorem{remark}{Remark}
\theoremstyle{definition}

\begin{document}
\title{A small morphism giving Abelian repetition threshold less than 2}
\author{James D.\ Currie \& Narad Rampersad\\
Department of Mathematics \&
Statistics\\
The University of Winnipeg\\
{\tt j.currie@uwinnipeg.ca, n.rampersad@uwinnipeg.ca}}
\maketitle

\begin{abstract}
It is known that there are infinite words over finite alphabets with Abelian repetition threshold arbitrarily close to 1; however, the construction previously used involves huge alphabets. In this note we give a short cyclic morphism (length 13) over an 8-letter alphabet yielding an Abelian repetition threshold less than 1.8. \end{abstract}
\section{Introduction}
A central topic in combinatorics on words is the construction of infinite words avoiding repetitions. For $1<k<2$, a {\bf $k$-power} is a word $z^k=xyx$, where $z=xy$, and $|xyx|/|xy|=k$. 

For an integer $n>1$, the {\bf repetitive threshold} RT$(n)$ is defined to be the infimum of $r$ such that 
there exists an infinite sequence over the $n$ letter alphabet not containing any $k$-power with $k>r.$ In 1972, Dejean \cite{dejean} conjectured that 
$$\mbox{RT}(n)=\left\{\begin{array}{ll}
7/4,&n=3\\
7/5,& n=4\\
n/(n-1),&n\ne 3,4
\end{array}\right.$$ 
This conjecture was verified in 2009 by the authors, and simultaneously by Micha\"{e}l Rao \cite{proof,rao}.

For $1<k<2$, an {\bf Abelian $k$-power} is a word $x_1yx_2$, where $x_1$ is an anagram of $x_2$, and $|x_1yx_2|/|x_1y|=k$.
Since 2003, one of the authors \cite{analog,probs2} has pointed out the desirability of formulating and proving the correct Abelian version of Dejean's conjecture; that is, giving the value of the {\bf Abelian threshold} ART$(n)$, where `$k$-power' in the definition of RT$(n)$ is replaced by `Abelian $k$-power'. Since Abelian squares -- words $x_1x_2$ where $x_1$ is an anagram of $x_2$ -- are unavoidable on 3 or fewer letters, one seeks ART$(n)$ where $n\ge 4$. In recent years, others \cite{petrova,art} have started to pay some attention to this problem, and we believe the time is ripe for a solution.

Cassaigne and the first author \cite{frac} showed  that $\lim_{n\rightarrow\infty} ART(n)=1$. However, to produce a word containing no Abelian $k$-power with $k\ge 1.8$ using their construction requires an alphabet of size on the order of $5^{485}$. In contrast to this, one of the authors \cite{analog} claimed, based on some computational evidence, that the fixed point of a certain small morphism on an 8-letter alphabet avoids such Abelian $k$-powers.  In this paper we prove this claim. Our methods (and the morphism) in question should generalize to give an upper bound on ART($n$).

\section{Preliminaries}
We use letters such as $u$, $v$, and $w$ to stand for {\bf words}, that is, finite strings of letters. For the empty word, containing no letters, we write $\epsilon$. We denote the set of all letters, or {\bf alphabet}, by $\Sigma$. The length of word $u$ is denoted by $|u|$, and if $a\in\Sigma$, then $|u|_a$ denotes the number of times letter $a$ appears in word $u$. Thus if $\Sigma=\{a,b,n\}$ and $u=banana$, we have $|u|=6$, $|u|_a=3$, $|u|_n=2$.

The set of all finite words over $\Sigma$ is denoted by $\Sigma^*$, and can be regarded as the free semigroup generated by $\Sigma$, where the operation is concatenation, and is written multiplicatively.  If $u$, $v\in\Sigma^*$ we say that $u$ is an {\bf anagram} of $v$, and write $u\sim v$ if for every letter $a\in\Sigma$ we have $|u|_a=|v|_a$. 
If $u, p, v, s\in\Sigma^*$ and $u=pvs$ we say that $p$, $v$, and $s$ are respectively a {\bf prefix}, a {\bf factor}, and a {\bf suffix} of $u$. The {\bf index} of $v$ in $u$ is given by $|p|$. (Thus a prefix has index 0.) 

A {\bf morphism} $f$ on $\Sigma^*$ is a semigroup homomorphism, and is determined by giving $f(a)$ for each letter $a\in\Sigma$. A morphism is {\bf $k$-uniform} if $|f(a)|=k$ for each $a\in\Sigma$. Iteration of morphisms is written using exponents; thus $f^2(u)=f(f(u))$, and so forth. We use the convention that $f^0(u)=u$.

If $a\in\Sigma$ and $f$ is a morphism such that $a$ is a prefix of $f(a)$, then for each non integer $n$ we have $f^n(a)$ is a prefix of $f^{n+1}(a)$. In the case that $|f(a)|>1$, we let the {\bf infinite word} $f^\omega(a)$ be the sequence which has $f^n(a)$ as a prefix for each $n$. It is usual to refer to  $f^\omega(a)$ as the fixed point of $f$ starting with $a$, and to extend notions such as prefix, factor, index, etc., to this infinite word in the obvious ways.

For the remainder of this paper, let $\Sigma=\{0,1,2,\ldots, 7\}.$ Let $\sigma:\Sigma^*\rightarrow\Sigma^*$  be the cyclic shift morphism where
$$\sigma(a)\equiv a+1 \,(\text{mod } 8)\text{ for }a\in\Sigma.$$ Consider the 13-uniform cyclic morphism $f:\Sigma^*\rightarrow\Sigma^*$ given by $$f(0)=0740103050260.$$ Thus for an integer $n$,
$$f(\sigma^n(0))= \sigma^n(f(0)).$$ We see that
\begin{align*}
|f(w)|&=13|w|\text{ for }w\in\Sigma^*\\
|f(w)|_k&=|w|+5|w|_k\text{ for }w\in\Sigma^*, k\in\Sigma.
\end{align*}
\section{Preimages and anagrams}
Let $X_1,X_2\in\Sigma^*$ such that $X_1\sim X_2$, and such that for $i=1,2$, we have 

\begin{align*}
X_i&=s_if(x_i)p_i\\
f(a_i)&=p_iq_i,\text{ some }a_i\in \Sigma\cup\{\epsilon\}, |p_i|\le 12,\\
f(b_i)&=r_is_i,\text{ some }b_i\in \Sigma\cup\{\epsilon\}, |s_i|\le 12.
\end{align*}

We can certainly parse words $X_1$ and $X_2$ this way if they are factors of $f^\omega(0)$ of length at least 12.

From $s_1 f(x_1)p_1\sim s_2f(x_2)p_2$ we have, in particular, that
$$|s_1 f(x_1)p_1|=|s_2f(x_2)p_2|,\text{ so}$$
$$|s_2p_2|-|s_1p_1|=|f(x_1)|-|f(x_2)|.$$

The definition of $f$ and the conditions on the $p_i$ and $s_i$ imply that for each $k\in\Sigma$, $|s_ip_i|_k\le 10$. For the remainder of this section we consider the words $X_i,x_i,p_i,q_i,r_i,s_i$ to be fixed. Define  $\Delta=|x_1|-|x_2|$.
Then 
\begin{align*}13|\Delta|&=|13|x_1|-13|x_2||\\
&=||f(x_1)|-|f(x_2)||\\
&=||s_2p_2|-|s_1p_1||\\
&\le 24.
\end{align*}
We conclude that $|\Delta|\le 1.$ In particular, $|x_1|\ge|x_2|-1$.

\begin{lemma}\label{a or b}Suppose that for some $k$ we have $|x_1|_k\ge |x_2|_k+1$. Then 
$a_2=k$ or $b_2=k$. 
\end{lemma}
\begin{proof}
We see that
\begin{align*}
|f(x_1)|_k&=|x_1|+5|x_1|_k\\
&\ge (|x_2|-1)+5(|x_2|_k+1)\\
&= |x_2|+5|x_2|_k+4\\
&=|f(x_2)|_k+4.
\end{align*}

Then
\begin{align*}
|s_2p_2|_k&=|s_2f(x_2)p_2|_k-|f(x_2)|_k\\
&=|s_1f(x_1)p_1|_k-|f(x_2)|_k\\
&\ge|s_1f(x_1)p_1|_k-(|f(x_1)|_k-4)\\
&=|s_1p_1|_k+4\\
&\ge 4.
\end{align*}
It follows that either $|s_2|_k\ge 2$ or $|p_2|_k\ge 2$. We conclude that $a_2=k$ or $b_2=k$. 
\end{proof}
It follows that there are at most 2 values $k$ for which $|x_1|_k>|x_2|_k$, and at most two values $j$ for which $|x_2|_j>|x_1|_j$.
\begin{lemma}\label{at least 2}
Suppose that for some $k$ we have  $|x_1|_k\ge|x_2|_k+2$.
Then $a_1x_1b_1\sim a_2x_2b_2$. 
\end{lemma}
\begin{proof}
Recall that $-1\le \Delta\le 1$. We will show that $\Delta=0$.
First of all, we have
\begin{align*}
|f(x_1)|_k&=|x_1|+5|x_1|_k\\
&\ge |x_1|+5(|x_2|_k+2)\\
&=|x_2|+\Delta+5(|x_2|_k+2)\\
&=|f(x_2)|_k+\Delta+10.
\end{align*}

Then
\begin{align*}
10&\ge |s_2p_2|_k\\
&=|s_1p_1|_k+|f(x_1)|_k-|f(x_2)|_k\\
&\ge|f(x_1)|_k-|f(x_2)|_k\\
&\ge\Delta+10.
\end{align*} 
and we deduce that $0\ge \Delta$. Also,
\begin{align}
\nonumber|s_2p_2|_k&=|s_1p_1|_k+|f(x_1)|_k-|f(x_2)|_k\\
\nonumber&\ge|f(x_1)|_k-|f(x_2)|_k\\
\nonumber&=|x_1|+5|x_1|_k-(|x_2|+5|x_2|_k)\\
\label{delta}&=\Delta+5(|x_1|_k-|x_2|_k)\\
\nonumber&\ge -1 + 5(2)\\ 
\nonumber=9.
\end{align}
The shortest suffix of $f(k)$ containing 4 $k$'s has length 8; the shortest suffix of $f(k)$ containing 5 $k$'s has length 10; no proper suffix of $f(k)$ contains more than 5 $k$'s; there is a single $k$ in $f(\ell)$ for $\ell\in\Sigma-\{k\}$. The analogous statements hold for prefixes of $f(k)$. Thus from $|s_2p_2|_k\ge 9$ we conclude that 
we have $a_2=b_2=k$, and 
either $|p_2|\ge 10$ and $|s_2|\ge 8$, 
or $|p_2|\ge 8$ and $|s_2|\ge 10$; in both cases $|s_2p_2|\ge 18$.

We claim that $\Delta\ge 0$, i.e., $|x_1|\ge|x_2|$. Otherwise, $|x_2|\ge |x_1|+1$ so that
\begin{align*}
24+|f(x_1)|&\ge |s_1f(x_1)p_1|\\
&=|s_2f(x_2)p_2|\\
&\ge 13|x_2|+18\\
&\ge 13(|x_1|+1)+18\\
&=31+|f(x_1)|,
\end{align*}

a contradiction. We earlier established that $\Delta \le 0$, so we now conclude that $\Delta=0$. From inequality (\ref{delta}), we see that $|s_2p_2|_k\ge 10$. This forces  $|p_2|, |s_2|\ge 10$.

Since $\Delta=0$, we have $|x_1|=|x_2|$. Since $|x_1|_k\ge |x_2|_k+2$, either there are letters $j_1\ne j_2$ such that $|x_2|_{j_i}>|x_1|_{j_i}$, or a single letter $j$ such that $|x_2|_j\ge |x_1|_j+2$. In the first case, by Lemma~\ref{a or b}, one of $a_1$ and $b_1$ is $j_1$, and the other $j_2$; in the second case $a_1=b_1=j$. In both cases, $b_1x_1a_1\sim b_2x_2a_2$. \end{proof}

\begin{lemma}\label{hats}
We can choose words $\hat{x}_i\in\{x_i,x_ib_i,a_ix_i,a_ix_ib_i\}$ for $i=1,2$ such that
$\hat{x}_1\sim\hat{x}_2$.
\end{lemma}
\begin{proof}
If for each $k\in\Sigma$ we have $|x_1|_k=|x_2|_k$, then let $\hat{x_i}=x_i$, and the result is true. Assume then that $|x_1|_k\ne |x_2|_k$ for some value $k$. Again, if for some $k\in\Sigma$ we have $||x_1|_k-|x_2|_k|\ge 2$, then the result is true by Lemma~\ref{at least 2}. Suppose then that $||x_1|_k-|x_2|_k|\le 1$ for each $k$. By Lemma~\ref{a or b}, $|x_1|_k>|x_2|_k$ for at most two values of $k$, and 
$|x_2|_k>|x_1|_k$ for at most two values of $k$.

Without loss of generality, suppose that $|x_1|\ge |x_2|$, so that $\Delta\ge 0$. Then $\Delta=0$ or $\Delta=1$.

\noindent{\bf Case 1: We have $\Delta =1$}.\\ Since
\begin{align*}\Delta&=|x_1|-|x_2|\\
&=\sum_{k\in\Sigma} |x_1|_k-\sum_{k\in\Sigma} |x_2|_k\\
&=\sum_{k\in\Sigma} (|x_1|_k-|x_2|_k)
\end{align*}
we must have $|x_1|_k> |x_2|_k$ for at least one value of $k$. By supposition, $||x_1|_k-|x_2|_k|\le 1$, so that
$|x_1|_k=1+ |x_2|_k$ for this value. By Lemma~\ref{a or b}, there can be at most two values of $k$ such that $|x_1|_k= 1+ |x_2|_k$. We make cases on the number of such $k$.

\noindent{\bf Case 1a: There is a unique $k_1\in\Sigma$ such that $|x_1|_{k_1}=1+ |x_2|_{k_1}$.}\\
By Lemma~\ref{a or b}, we have $a_2=k_1$ or $b_2=k_1$. 
Let $\hat{x}_1=x_1$ and 
$$\hat{x}_2=
\begin{cases}a_2x_2, &\text{ if }a_2=k_1;\\
x_2b_2, &\text{ otherwise. (In this case }b_2=k_1.)
\end{cases}
$$

We have 
\begin{align*}1&=\Delta\\
&=|x_1|-|x_2|\\
&=\sum_{k\in\Sigma} (|x_1|_k- |x_2|_k)\\
&= 1 +\sum_{k\in\Sigma-\{k_1\}}(|x_1|_k- |x_2|_k)
\end{align*}
Then $\sum_{k\in\Sigma-\{k_1\}}(|x_1|_k- |x_2|_k)=0$. However, each term of this sum is non-positive, since $|x_1|_k- |x_2|_k<1$ for $k\ne k_1$. It follows that no term of this sum can be negative, so that $|x_1|_k=|x_2|_k$ whenever $k\ne k_1$. It follows that
$$\hat{x}_1=x_1\sim kx_2\sim\hat{x}_2\text{, as desired.}$$

\noindent{\bf Case 1b: There are two values $k_1\ne k_2\in\Sigma$ such that for $i=1,2$, $|x_1|_{k_i}=1+ |x_2|_{k_i}$.}\\
By Lemma~\ref{a or b} we have $\{a_2,b_2\}=\{k_1,k_2\}$. Let $\hat{x_2}=a_2x_2b_2$. From
\begin{align*}1&=\Delta\\
&=|x_1|-|x_2|\\
&=\sum_{k\in\Sigma} (|x_1|_k- |x_2|_k)\\
&= 2 +\sum_{k\in\Sigma-\{k_1,k_2\}}(|x_1|_k- |x_2|_k)
\end{align*}
and the fact that $|x_1|_k- |x_2|_k\le 0$ for $k\ne k_1,k_2$, we conclude that there is exactly one value $k_3$ such that $|x_2|_{k_3}>|x_1|_{k_3}$. Then $|x_2|_{k_3}=1+|x_1|_{k_3}$, and by Lemma~\ref{a or b} we conclude that $k_3\in\{a_1,b_1\}$.
Let $$\hat{x}_1=
\begin{cases}a_1x_1, &\text{ if }a_1=k_3;\\
x_1b_1, &\text{ otherwise. (In this case }b_1=k_3.)
\end{cases}
$$
Then $\hat{x_1}\sim\hat{x}_2$, as desired.

\noindent{\bf Case 2: We have $\Delta =0$}.\\ By Lemma~\ref{a or b}, the number of $k$ such that $|x_1|_k>|x_2|_k$ is 1, or 2. (We are assuming there is at least one $k$ where these differ.) Since
\begin{align*}0&=\Delta\\
&=\sum_{k\in\Sigma} (|x_1|_k- |x_2|_k),
\end{align*}
there will be the same number of $k$ such that $|x_2|_k>|x_1|_k$.\\
\noindent{\bf Case 2a: There is a unique $k_1\in\Sigma$ such that $|x_1|_{k_1}=1+ |x_2|_{k_1}$ and a unique $k_2\in\Sigma$ such that $|x_2|_{k_2}=1+ |x_1|_{k_2}$}\\
By Lemma~\ref{a or b}, we have $a_2=k_1$ or $b_2=k_1$ and $a_1=k_2$ or $b_1=k_2$. 
Let 
$$\hat{x}_2=
\begin{cases}a_2x_2, &\text{ if }a_2=k_1;\\
x_2b_2, &\text{ otherwise. (In this case }b_2=k_1.)
\end{cases}
$$
Let 
$$\hat{x}_1=
\begin{cases}a_1x_1, &\text{ if }a_1=k_2;\\
x_1b_1, &\text{ otherwise. (In this case }b_1=k_2.)
\end{cases}
$$

\noindent{\bf Case 2b: There are distinct $k_1,k_2,k_3,k_4\in\Sigma$ such that 
\begin{align*}
|x_1|_{k_1}=1+ |x_2|_{k_1}\\
|x_1|_{k_2}=1+ |x_2|_{k_2}\\
|x_2|_{k_3}=1+ |x_1|_{k_3}\\
|x_2|_{k_4}=1+ |x_1|_{k_4}
\end{align*}}\\
By Lemma~\ref{a or b}, we have $\{a_2,b_2\}=\{k_1,k_2\}$ and
$\{a_1,b_1\}=\{k_3,k_4\}$.
Let $\hat{x}_1=a_1x_1b_1$,  $\hat{x}_2=a_2x_2b_2$. Then $\hat{x}_1\sim\hat{x}_2$.
\end{proof}
\section{Preimages of Abelian powers}
If $k\in\Sigma$, then the index of $f(k)$ in $f^\omega(0)$ is always a multiple of 13.
Suppose $f^n(0)$ contains a prefix $PX_1YX_2$ with 
$X_1\sim X_2$, $|X_1|> 24$. For $i=1,2$, write 

\begin{align*}
X_i&=s_if(x_i)p_i\\
f(a_i)&=p_iq_i,\text{ some }a_i\in \Sigma\cup\{\epsilon\}, |p_i|\le 12,\\
f(b_i)&=r_is_i,\text{ some }b_i\in \Sigma\cup\{\epsilon\}, |s_i|\le 12.
\end{align*}

Since each of the $|X_i|> 24$, we conclude that each $|x_i|$ is non-empty. It follows that the indices of $f(x_1)$ and $f(x_2)$ are multiples of 13, which forces $|p_1Ys_2|$ to be a multiple of 13, and $|p_1s_2|$ to be a multiple of 13. We therefore write $PX_1YX_2=f(Qb_1x_1a_1yb_2x_2)q_2$ where $Qb_1x_1a_1yb_2x_2a_2$ is a prefix of $f^{n-1}(0)$, 
$P=f(Q)r_1$, $Y=q_1f(y)r_2$, and $y\in\Sigma^*$. Note that the $a_i,b_i,p_i,q_i,r_i,s_i,y$ may all be empty.

\begin{lemma}\label{previous power}
Suppose $n>1$ and $f^\omega(0)$ contains an Abelian power $X_1YX_2$ with 
$X_1\sim X_2$, $|X_1|\ge 24n$. Then $f^\omega(0)$ contains an Abelian power $\hat{x}_1\hat{y}\hat{x}_2$ with $\hat{x}_1\sim\hat{x}_2$, $$|X_1Y|\ge \frac{13n}{n+1}|\hat{x}_1\hat{y}|\text{ and }$$
$$\frac{|X_1YX_2|}{|X_1Y|}\le\frac{|\hat{x}_1\hat{y}\hat{x}_2|}{|\hat{x}_1\hat{y}|}+\frac{72}{|X_1Y|}.$$
\end{lemma}
\begin{proof}
Choose $\hat{x}_1$ and $\hat{x}_2$ as in Lemma~\ref{hats}. Then $\hat{x_1}\sim\hat{x_2}.$ Let 
$$\hat{y}=
\begin{cases}y,&\text{ if }\hat{x}_1\in\{b_1x_1a_1,x_1a_1\},\hat{x}_2\in\{b_2x_1a_2,b_2x_2\};\\
a_1y,&\text{ if }\hat{x}_1\in\{b_1x_1,x_1\},\hat{x}_2\in\{b_2x_1a_2,b_2x_2\};\\
yb_2,&\text{ if }\hat{x}_1\in\{b_1x_1a_1,x_1a_1\},\hat{x}_2\in\{x_1a_2,x_2\};\\
a_1yb_2,&\text{ if }\hat{x}_1\in\{b_1x_1,x_1\},\hat{x}_2\in\{x_1a_2,x_2\}.
\end{cases}
$$
Then $f^{n-1}(0)$ has prefix $Qb_1x_1a_1yb_2x_2a_2$ as in the previous discussion, which contains the factor $\hat{x}_1\hat{y}\hat{x}_2$ by our choice of $\hat{x}_1,\hat{y}$, and $\hat{x}_2$.

We note that 
\begin{align*}
|\hat{x}_1\hat{y}|&\le|b_1x_1a_1yb_2|\\
&=\frac{|f(b_1x_1a_1yb_2)|}{13}\\
&=\frac{|X_1Y|+|r_1s_2|}{13}\\
&\le\frac{|X_1Y|+24}{13}\\
&\le\frac{|X_1Y|\frac{n+1}{n}}{13}
\end{align*}
Thus $$|X_1Y|\ge \frac{13n}{n+1}|\hat{x}_1\hat{y}|.$$
Also,
\begin{align*}
|\hat{x}_1\hat{y}\hat{x}_2|&\ge|x_1a_1yb_2x_2|\\
&=\frac{|f(x_1a_1yb_2x_2)|}{13}\\
&=\frac{|X_1YX_2|-|s_1p_2|}{13}\\
&\ge\frac{|X_1YX_2|-24}{13}
\end{align*}

Therefore,
\begin{align*}
\frac{|X_1YX_2|}{|X_1Y|}-\frac{|\hat{x}_1\hat{y}\hat{x}_2|}{|\hat{x}_1\hat{y}|}&\le
\frac{|X_1YX_2|}{|X_1Y|}-\frac{\frac{|X_1YX_2|-24}{13}}{\frac{|X_1Y|+24}{13}}\\
&=
\frac{|X_1YX_2|}{|X_1Y|}-\frac{|X_1YX_2|-24}{|X_1Y|+24}\\
&=
\frac{|X_1YX_2|(\stkout{|X_1Y|}+24)-|X_1Y|(\stkout{|X_1YX_2|}-24)}{|X_1Y|(|X_1Y|+24)}\\
&\le
\frac{2\stkout{|X_1Y|}24+\stkout{|X_1Y|}24}{\stkout{|X_1Y|}(|X_1Y|+24)}\\
&=
\frac{72}{|X_1Y|+24}\\
&\le\frac{72}{|X_1Y|},\text{ as desired}.
\end{align*}
\end{proof}
\begin{corollary}\label{bound}
Suppose $n>1$ and $c\in {\mathbb R}$ are such that whenever $x_1yx_2$ is an Abelian power in $f^\omega(0)$ with $x_1\sim x_2$ and $|x_1|\le 24n$,
then $\frac{|x_1yx_2|}{|x_1y|}\le c$. Then if $X_1YX_2$ is an Abelian power in $f^\omega(0)$ we have 
$$\frac{|X_1YX_2|}{|X_1Y|}\le c+ \frac{k}{1-r},$$
where $k=\frac{3}{n}$ and $r=\frac{n+1}{13n}$
\end{corollary}
\begin{proof}Repeatedly using Lemma~\ref{previous power}, create a sequence of Abelian powers
$$\left(X_1^{(i)}Y^{(i)}X^{(i)}_2\right)_{i=0}^m$$
such that $$X_1^{(0)}=X_1, Y^{(0)}=Y, X^{(0)}_2=X_2,$$
$$|X_1^{(m)}|\le 24n$$
$$|X_1^{(i)}|> 24n,\text{ for }i=1,2,\ldots, m-1$$ 
$$X_1^{(i+1)}\sim X_2^{(i+1)}\text{ for }i=0,1,\ldots, m-1$$
 \begin{equation}\label{decrease}|X_1^{(i)}Y^{(i)}|\ge \frac{13n}{n+1}|X_1^{(i+1)}Y^{(i+1)}|\end{equation}
$$\frac{|X_1^{(i)}Y^{(i)}X_2^{(i)}|}{|X_1^{(i)}Y^{(i)}|}\le\frac{|X_1^{(i+1)}Y^{(i+1)}X_2^{(i+1)}|}{|X_1^{(i+1)}Y^{(i+1)}|}+\frac{72}{|X_1^{(i)}Y^{(i)}|},\text{ for }i=0,1,\ldots, m-1.$$
Property (\ref{decrease}) coming from Lemma~\ref{previous power} ensures that $m$ is finite.
We see that for $0\le i\le m-1$
$$\frac{1}{|X_1^{(i)}Y^{(i)}|}\le \frac{r^i}{|X_1^{(0)}Y^{(0)}|}.$$
Since $k=\frac{3}{n}\ge \frac{72}{|X_1^{(0)}Y^{(0)}|}$, this means that
$$\frac{|X_1^{(0)}Y^{(0)}X_2^{(0)}|}{|X_1^{(0)}Y^{(0)}|}\le\frac{|X_1^{(m)}Y^{(m)}X_2^{(m)}|}{|X_1^{(m)}Y^{(m)}|}+k\sum_{i=1}^mr^i\le  c+ \frac{k}{1-r}.$$
\end{proof}
\section{Search depth}
\begin{remark}\label{length 2} Say that word $u$ is {\em equivalent} to word $v$ if $u=\sigma^i(v)$ for some integer $i$. Any length 2 factor $v$ of $f^\omega(0)$ is 
equivalent to a factor of 07401030502, the length 11 prefix of $f^\omega(0)$ Since $f$ is cyclic, if $u$ is equivalent to $v$, then every factor of $f(u)$ is equivalent to a factor of $f(v)$. \end{remark}
Let $$g(x)=\left\lfloor \frac{x+24}{13}\right\rfloor.$$
\begin{lemma}\label{13p}
Suppose that $u$ is factor of $f^\omega(0)$, and $g^t(|u|)\le 2$. Then $u$ is equivalent to a factor of the prefix of $f^\omega(0)$ of length $13^t(11)$.
\end{lemma}
\begin{proof}
Create a sequence of words $u=u_0$, $u_1,\ldots u_t$ such that $u_i$ is a factor of $f(u_{i+1})$ for each $i$, and $u_{i+1}$ is a factor of $f^\omega(0)$ with $|f(u_{i+1})|\le |u_i|+24$. Then 
$$u_0\text{ is a factor of }f^t(u_t)\text{ and }$$
$$|u_t|\le g^t(|u_0|)\le 2.$$ 
It follows that $u_t$ is equivalent to a factor of 07401030502, the prefix of $f^\omega(0)$ of length 11. Therefore, $u=u_0$ is equivalent to a factor 
of $f^t(07401030502)$, which has length $13^t(11).$
\end{proof}
\begin{lemma}\label{13^5} Suppose that $f^\omega(0)$ contains a factor 
$X_1YX_2$ such that
\begin{align*} 
X_1&\sim X_2\\
\frac{|X_1YX_2|}{|X_1Y|}&>1.713\\
|X_1|&\le 1000.
\end{align*}
Then $f^5(0)$ contains such a factor.
\end{lemma}
\begin{proof} For such a word, we have
\begin{align*}
|X_1YX_2|&> 1.713|X_1Y|\\
&=1.713(|X_1YX_2|-|X_2|)\\
&=1.713(|X_1YX_2|-|X_1|)
\end{align*}
so that
\begin{align*}
1.713|X_1|>0.713|X_1YX_2|
\end{align*}
and
\begin{align*}
|X_1YX_2|&<\frac{1.713|X_1|}{0.713}\\
&\le \frac{1713}{0.713}\\
&<2403.
\end{align*}
We find that $g^5(2403)=2$.
The result follows from Lemma~\ref{13p}. 
\end{proof}
\section{Search result and bound}
We searched the word $f^5(0)$. All its factors $X_1YX_2$ with
$X_1\sim X_2$ and $|X_1|\le 1000$ obey
$$\frac{|X_1YX_2|}{|X_1Y|}\le \frac{841}{491}\approx 1.71283.$$
We conclude from Lemma~\ref{13^5} that $f^\omega(0)$ contains no factor $X_1YX_2$ with
\begin{align*} 
X_1&\sim X_2\\
\frac{|X_1YX_2|}{|X_1Y|}&>1.713\\
|X_1|&\le 1000.
\end{align*}
\begin{theorem} Word $f^\omega(0)$ contains no factor $X_1YX_2$ 
with $X_1\sim X_2$ and $$\frac{|X_1YX_2|}{|X_1Y|}>\frac{876775}{489527}\approx 1.79107.$$ 
\end{theorem}
\begin{proof} This follows from Corollary~\ref{bound}, letting $n=1000/24$ and $c=\frac{841}{491}$.
\end{proof}

\end{document}